\documentclass[12pt]{article}
\usepackage[T2A]{fontenc}
\usepackage[cp1251]{inputenc}
\usepackage[english]{babel}
\usepackage{amssymb,amsmath,amsthm}
\usepackage{cite}
\usepackage[left=3cm,right=2cm, top=3cm,bottom=3cm,bindingoffset=0cm]{geometry}

\usepackage{xcolor}

\theoremstyle{plain}
\newtheorem{theorem}{Theorem}[section]
\newtheorem{corollary}{Corollary}[theorem]

\begin{document}
\title{On a Finite Group Generated by Subnormal Supersoluble Subgroups}
\author{Victor S. Monakhov\footnote{
Department of Mathematics and Programming Technologies, Francisk Skorina Gomel State University, Belarus,
victor.monakhov@gmail.com}}
\date{}
\maketitle

{\small
\leftskip=1cm \rightskip=1cm

\noindent\textbf{Abstract.} 
Supersolubility of a finite group~$G=\langle A,B\rangle$
with the nilpotent derived subgroup~$G^\prime$ is established under
the condition that the subgroups~$A$ and $B$ are both subnormal
and supersoluble.

\medskip


\medskip


}

\section{Main Result}

All groups in this paper are finite.
We use the standard notations and terminology of~\cite{hup}.

B.~Huppert~\cite{hup53} and R.~Baer~\cite{b57} gave the first
examples of nonsupersoluble groups
that were a product of normal supersoluble subgroups.
R.~Baer~\cite{b57} established supersolubility of a group~$G=AB$
with the nilpotent derived subgroup~$G^\prime$
and normal supersoluble subgroups~$A$ and~$B$.
A.\,F.~Vasil'ev and T.\,I.~Vasil'eva~\cite{vv} showed that
instead of nilpotency of the derived subgroup $G'$
it is enough to require nilpotency of the $\mathcal A$-residual,
where~$\mathcal A$ is the formation of all groups with abelian Sylow subgroups.
In these results, normality of subgroups~$A$ and~$B$
can be weakened to subnormality of $A$ and $B$~\cite{mch}.
This themes was developed in many papers,
see for example~\cite{b}.

If a group~$G$ is generated by subnormal subgroups~$X$ and~$Y$,
then it is not always true that $G=XY$.
The simplest examples are the dihedral group of order~8
and nonabelian group of order~$p^3$ and exponent~$p$.
The normal closure~$X^G$ of a subnormal supersoluble subgroup~$X$ in a group $G$
can be a nonsupersoluble subgroup.
For example, $G=PSU_3(2)\rtimes C_2$~\cite[SmallGroup(144,182)]{gap}
contains a nonsupersoluble maximal subgroup
$H=S_3\wr C_2$~\cite[SmallGroup(72,40)]{gap} of index~2.
A subgroup~$X=S_3\times S_3$ of~$H$ is supersoluble and not normal in~$G$.
Since~$|H:X|=2$, $X^G=H$ and~$X$ is subnormal in~$G$.

Next, we write $G=\langle A,B\rangle$ when a group~$G$ is generated by subgroups~$A$ and~$B$.

In this paper, we prove the following theorem.

\begin{theorem}
Let~$A$ and~$B$ be subnormal supersoluble subgroups of a group~$G$
and let~$G=\langle A,B\rangle$.
Then $G$ is metanilpotent and has a Sylow tower of supersoluble type.
Furthermore, $G$ is supersoluble when one of the following conditions holds.

$(1)$ $G^\mathcal{A}$ is nilpotent.

$(2)$ $(|A:A^\prime |,|B:B^\prime |)=1$.
\end{theorem}

\begin{proof}
Let $p\in \pi (G)$ and let $A_p$ and $B_p$ be Sylow $p$-subgroups of $A$ and~$B$, respectively.
Suppose that $p$ is the largest prime in~$\pi (G)$.
Then $A_p$ and $B_p$ are normal in~$A$ and~$B$, respectively~\cite[VI.9.1]{hup}.
Therefore $A_p$ and $B_p$ are subnormal in~$G$ and
$\langle A_p,B_p\rangle \le O_p(G)$.
By induction, $G/O_p(G)$ has a Sylow tower of supersoluble type.
Since $AO_p(G)/O_p(G)$ and $BO_p(G)/O_p(G)$ are subnormal $p^\prime$-subgroups of  $G/O_p(G)$,
$$
G/O_p(G)=\langle AO_p(G)/O_p(G),BO_p(G)/O_p(G)\rangle
$$
is a $p^\prime$-group. Hence $G$ has a Sylow tower of supersoluble type.

Since~$A^\prime $ is nilpotent~\cite[VI.9.1]{hup} and subnormal in~$G$,
we have $A^\prime \le F(A)\le F(A)^G\le F(G)$.
Consequently,
$$
AF(G)/F(G)\cong A/(A\cap F(G))
$$
is abelian and subnormal in~$G/F(G)$. Similarly,
$$
B^\prime \le F(B)\le F(B)^G\le F(G),\ BF(G)/F(G)\cong B/(B\cap F(G)),
$$
therefore $BF(G)/F(G)$ is abelian and subnormal in~$G/F(G)$.
It is clear that
$$
G/F(G)=\langle AF(G)/F(G),BF(G)/F(G)\rangle. \eqno (\ast)
$$
So,
$(AF(G)/F(G))^{G/F(G)}$ and~$(BF(G)/F(G))^{G/F(G)}$ are
normal in $G/F(G)$ and nilpotent. Hence
$$
G/F(G)=(AF(G)/F(G))^{G/F(G)}(BF(G)/F(G))^{G/F(G)}
$$
is nilpotent, and $G$ is metanilpotent.

(1) Let $G^\mathcal{A}$ be nilpotent.
Then $G^\mathcal{A}\le F(G)$.
Since~$G$ is metanilpotent, we conclude that $G/F(G)$ is nilpotent,
hence~$G/F(G)$ is abelian and~$G^\prime \le F(G)$.
Therefor~$AG^\prime$ and~$BG^\prime$ is normal in~$G$.
According to~\cite[Lemma~10]{mch},
$AG^\prime $ and~$BG^\prime $ are supersoluble.
Hence $G=(AG^\prime )(BG^\prime )$ is supersoluble by Baer's Theorem.

(2) Let  $(|A:A^\prime |,|B:B^\prime |)=1$.
Since $A^\prime B^\prime \subseteq F(G)$, we get
$$
(|AF(G)/F(G)|,|BF(G)/F(G)|)=1.
$$
From~$(\ast)$, it follows that
$$
G/F(G)=AF(G)/F(G)\times BF(G)/F(G)
$$
is abelian and~$G^\prime \le F(G)$.
Therefore $G=(AG^\prime )(BG^\prime )$ is supersoluble.
\end{proof}

\begin{corollary}
Let~$A$ and~$B$ be subnormal supersoluble subgroups of a group~$G$
and let~$G=\langle A,B\rangle$. If the derived subgroup of $G$ is nilpotent,
then~$G$ is supersoluble.
\end{corollary}

\end{document}